\documentclass[english,reqno]{amsart}%
\usepackage[T1]{fontenc}
\usepackage[latin9]{inputenc}
\usepackage{amsthm}
\usepackage{amstext}
\usepackage{amssymb}
\usepackage{amsfonts}
\usepackage{babel}
\usepackage{amsmath}
\usepackage{graphicx}
\usepackage{setspace}%
\providecommand{\U}[1]{\protect\rule{.1in}{.1in}}
\makeatletter
\numberwithin{equation}{section}
\numberwithin{figure}{section}
\newtheorem{theorem}{Theorem}
\theoremstyle{plain}

\newtheorem{corollary}{Corollary}

\newtheorem{definition}{Definition}

\newtheorem{lemma}{Lemma}

\begin{document}
\title{On local property of absolute summability of factored Fourier series}
\author{Hüsey\.in Bor}
\address{P. O. Box 121, 06502 Bah{ç}elievler, Ankara, Turkey.\\
E-mail:hbor33@gmail.com}
\author{Dansheng Yu}
\address{Department of Mathematics, Hangzou Normal University, Hangzhou, Zhejiang
310036, China.\\
E-mail:danshengyu@yahoo.com.cn}
\author{Ping Zhou}
\address{Department of Mathematics, statistics and Computer Science, \\
St. Francis Xavier University, Antigonish, Nova Scotia, Canada B2G 2W5. \\
E-mail:pzhou@stfx.ca}
\thanks{Research of the second author is supported by NSF of China
(10901044), and Program for Excellent Young Teachers in HZNU.
Research of the third author is supported by NSERC of Canada.}

\begin{abstract}
We establish two general theorems on the local properties of the absolute
summability of factored Fourier series by applying a recently defined absolute
summability, $\left\vert A,\alpha_{n}\right\vert _{k}$ summability, and the
class $\mathcal{S}\left(  \alpha_{n},\phi_{n}\right)  $, which generalize some
well known results and can be applied to improve many classical absolute
summability methods.

\end{abstract}
\maketitle

\section{Introduction}

Let $A:=\left(  a_{nk}\right)  $ be a lower triangular matrix and $\left\{
s_{n}\right\}  $ the partial sums of $\sum a_{n}.$ Let $\left\{  \alpha
_{n}\right\}  $ be a nonnegative sequence,$\ $then the series $\sum a_{n}$ is
said to be summable $\left\vert A,\alpha_{n}\right\vert _{k},k\geq1,$ if (see
\cite{Yu})
\[
\sum_{n=1}^{\infty}\alpha_{n}\left\vert A_{n}-A_{n-1}\right\vert ^{k}<\infty,
\]
where
\[
A_{n}:=\sum_{v=1}^{n}a_{nv}s_{v}.
\]

In particular, if $\alpha_{n}=n^{k-1}$ $,\ $then $\left\vert A,\alpha
_{n}\right\vert _{k}-$summability reduces to the $\left\vert A\right\vert
_{k}$-summability (see \cite{Ta}). Let $A$ be the Cesàro matrices $C:=\left(
c_{nv}\right)  $ of order $\alpha,$ that is,%
\[
c_{nv}:=\frac{A_{n-v}^{\alpha-1}}{A_{n}^{\alpha}},\ v=0,1,\cdots,n,
\]
where%
\[
A_{n}^{\alpha}:=\frac{\Gamma\left(  n+\alpha+1\right)  }{\Gamma\left(
\alpha+1\right)  \Gamma\left(  n+1\right)  },\ n=0,1,\cdots.
\]
When $\alpha_{n}=n^{\delta k+k-1},k\geq1,\delta\geq0,$ $\left\vert
A,\alpha_{n}\right\vert _{k}-$summability is usually called $\left\vert
C,\alpha,\delta\right\vert _{k}$-summability. Therefore, a series $\sum a_{n}$
is said to be summable $\left\vert C,\alpha,\delta\right\vert _{k},\ k\geq
1$,$\ \alpha>-1,$\ if (see \cite{F})%
\[
\sum_{n=1}^{\infty}n^{\delta k+k-1}\left\vert \sigma_{n}^{\alpha}-\sigma
_{n-1}^{\alpha}\right\vert ^{k}<\infty,
\]
where%
\[
\sigma_{n}^{\alpha}:=\sum_{j=0}^{n}\frac{A_{n-j}^{\alpha-1}}{A_{n}^{\alpha}%
}s_{j}.
\]

For any positive sequence $\left\{  p_{n}\right\}  $ such that $P_{n}%
=p_{0}+p_{1}+\cdots+p_{n}\rightarrow\infty$, the corresponding Riesz matrix
$R$ has the entries
\[
r_{nv}:=\frac{p_{v}}{P_{n}},\ \ \,v=0,1,\cdots,n,\ \,n=0,1,2,\cdots.
\]
Taking $\alpha_{n}=\left(  \frac{P_{n}}{p_{n}}\right)  ^{\delta k+k-1}$ and
$\alpha_{n}=n^{\delta k+k-1}$, we get two special absolute summability,
$\left\vert \overline{N},p_{n},\delta\right\vert _{k}$ summability and
$\left\vert R,p_{n},\delta\right\vert _{k}$ summability, of $\left\vert
R,\alpha_{n}\right\vert _{k}$ summability, respectively. In particular, if
$np_{n}\asymp P_{n}$, then $\left\vert \overline{N},p_{n},\delta\right\vert
_{k}$ summability and $\left\vert R,p_{n},\delta\right\vert _{k}$ summability
are equivalent. See \cite{Bor01} and \cite{Bor02} for more details on
$\left\vert \overline{N},p_{n},\delta\right\vert _{k}$ summability and
$\left\vert R,p_{n},\delta\right\vert _{k}$ summability.

One can find more examples of $\left\vert A,\alpha_{n}\right\vert _{k}%
-$summability for different weight sequences $\left\{  \alpha_{n}\right\}  $
and different summability matrices $A$ discussed in many papers, see
\cite{Bor01}, \cite{Bor02}, \cite{Borwein}, \cite{Leindler}, and
\cite{Sulaiman} for examples.

Let $f$ be a function with period $2\pi,$ integrable ($L$) over $(-\pi,\pi).$
Without loss of generality we may assume that the constant term in the Fourier
series of $f\left(  t\right)  $ is zero, so that%
\[
\int_{-\pi}^{\pi}f\left(  t\right)  dt=0
\]
and%
\[
f\left(  t\right)  \sim%
{\displaystyle\sum\limits_{n=1}^{\infty}}
\left(  a_{n}\cos nt+b_{n}\sin nt\right)  \equiv%
{\displaystyle\sum\limits_{n=1}^{\infty}}
C_{n}\left(  t\right)  .
\]
It is well known that (see \cite{T}) the convergence of the Fourier series at
$t=x$ is a local property of the generating function $f\left(  t\right)  $
(i.e., it depends only on the behavior of $f$ in a arbitrarily small
neighborhood of $x$), and hence the summability of the Fourier series at $t=x$
by any regular linear summability method is also a local property of the
generating function $f\left(  t\right)  .$

In 1939, Bosanquet and Kestleman (\cite{BK}) showed that even the summability
$\left\vert C,1\right\vert $ of the Fourier series at a point is not a local
property of $f.$ Mohanty (\cite{M}) subsequently observed that the summability
$\left\vert R,\log n,1\right\vert $ of the factored series%
\[
\sum C_{n}\left(  t\right)  /\log\left(  n+1\right)  ,
\]
at any point is a local property of $f,$ whereas the summability $\left\vert
C,1\right\vert $ of this series is not. Several generalizations of Mohanty's
result have been made by many authors, for examples, see, Bhatt (\cite{Bh}),
Bor (\cite{Bor02}-\cite{Bor4}), Borwein (\cite{Borwein2}), Sarigöl
(\cite{Sa1}, \cite{Sa2}), etc.

For any lower triangular matrix $A,$ associated it with two lower triangular
matrices $\overline{A}$ and $\widehat{A}$ defined by
\[
\overline{a}_{nv}=\sum_{r=v}^{n}a_{nr},\ v=0,1,2,\cdots,n\text{ and
}n=0,1,2,\cdots,
\]
and%
\[
\widehat{a}_{nv}=\overline{a}_{nv}-\overline{a}_{n-1,v},\ v=0,1,\cdots
,n-1;n=1,2,3,\cdots.\ \widehat{a}_{nn}=a_{nn}=\overline{a}_{nn}.
\]

Sarigöl (\cite{Sa2}) proved the following theorem:

\ \ \ \ \ \newline\textbf{Theorem A.} \textit{Let} $A$\textit{ be a lower
triangular matrix with nonnegative entries satisfying}

(i) $a_{n-1,v}\geq a_{nv}$ for $n\geq v+1,$

(ii) $\overline{a}_{n0}=1,\ n=0,1,\cdots,$

(iii) $\sum_{v=1}^{n-1}a_{vv}\widehat{a}_{n,v+1}=O\left(  a_{nn}\right)  ,$

(iv) $\Delta X_{n}=O\left(  \frac{1}{n}\right)  ,\ X_{n}=\left(
na_{nn}\right)  ^{-1},n=1,2,\cdots,X_{0}=0.$\newline\textit{If }$\left\{
\theta_{n}\right\}  $\textit{ holds for the following conditions,}

\textit{(v) }$\sum_{v=1}^{\infty}\left(  \theta_{v}a_{vv}\right)  ^{k-1}%
X_{v}^{k-1}\frac{1}{v}\lambda_{v}^{k}<\infty,$

\textit{(vi) }$\sum_{v=1}^{\infty}\left(  \theta_{v}a_{vv}\right)  ^{k-1}%
X_{v}^{k}\Delta\lambda_{v}<\infty,$

\textit{(vii) }$\sum_{n=v+1}^{\infty}\left(  \theta_{n}a_{nn}\right)
^{k-1}\left\vert \Delta\widehat{a}_{n,v+1}\right\vert =O\left(  \left(
\theta_{v}a_{vv}\right)  ^{k-1}a_{vv}\right)  $ \newline\textit{and}

(viii) $\sum_{n=v+1}^{\infty}\left(  \theta_{n}a_{nn}\right)  ^{k-1}%
\widehat{a}_{n,v+1}=O\left(  \left(  \theta_{v}a_{vv}\right)  ^{k-1}\right)
,$

\noindent\textit{then the summability of }$\left\vert A,\theta_{n}%
^{k-1}\right\vert _{k},k\geq1,$ of the\textit{ series }$\sum\lambda_{n}%
X_{n}C_{n}\left(  t\right)  $\textit{ at any point is a local property of
}$f,$\textit{where }$\left\{  \lambda_{n}\right\}  $ \textit{is a convex
sequence such that }$\sum n^{-1}\lambda_{n}$\textit{ is convergent.}

\ Theorem A generalized some well known results on the local property of
summability of factored Fourier series. Although, there are some matrices
satisfying the conditions in Theorem A, a Cesàro's matrix may not satisfy all
the conditions (i)-(iii). In fact, (ii) and (iii) do not hold for any
$\alpha>1$ or $\alpha<1$. Furthermore, Rhaly's generalized Cesàro matrices and
the $p-$Cesàro matrices do not satisfy the conditions of Theorem A neither
(see Section 3 for the definitions of Rhaly's generalized Cesàro matrices and
the $p-$Cesàro matrices).

In the present paper, we establish a new factor theorem which generalizes
Theorem A, and can be applied to many well known matrices, including the ones
mentioned above. We need the following class of matrices, $\mathcal{S}\left(
\alpha_{n},\phi_{n}\right)  ,$ which is recently introduced by Yu and Zhou
(\cite{YZ}):

\begin{definition}
Let$\ \left\{  \alpha_{n}\right\}  ,\ \left\{  \phi_{n}\right\}  \ $ be
sequences of positive numbers. We say that a lower triangular matrix
$A:=\left(  a_{nk}\right)  \in\mathcal{S}\left(  \alpha_{n},\phi_{n}\right)
,$ if it satisfies the following conditions
\begin{equation}
\sum_{i=0}^{n-1}\left\vert \Delta_{i}\widehat{a}_{ni}\right\vert =O\left(
\phi_{n}\right)  ; \tag{T1}\label{t1}%
\end{equation}%
\begin{equation}
\left\vert \widehat{a}_{ni}\right\vert =O\left(  \phi_{n}\right)
,\ i=0,1,\cdots,n; \tag{T2}\label{t2}%
\end{equation}%
\begin{equation}
\sum_{n=i+1}^{\infty}\alpha_{n}\phi_{n}^{k-1}\left\vert \Delta_{i}%
\widehat{a}_{ni}\right\vert =O\left(  \alpha_{i}\phi_{i}^{k}\right)  ;
\tag{T3}\label{t3}%
\end{equation}%
\begin{equation}
\sum_{n=i+1}^{\infty}\alpha_{n}\phi_{n}^{k-1}\left\vert \widehat{a}%
_{n,i+1}\right\vert =O\left(  \alpha_{i}\phi_{i}^{k-1}\right)  .
\tag{T4}\label{t4}%
\end{equation}

\end{definition}

Our main results are the following:

\begin{theorem}
Let $\ \left\{  \alpha_{n}\right\}  ,$ and$\ \left\{  \phi_{n}\right\}  $ be
sequences of positive numbers$.\ $Let$\ \left\{  \lambda_{n}\right\}  \in BV$
be a sequence of complex numbers\footnote{We say a sequence of complex numbers
$\left\{  \lambda_{n}\right\}  \in BV,$ if $\sum_{n=1}^{\infty}\left\vert
\Delta\lambda_{n}\right\vert <\infty.$} such that \ $\lambda_{n+1}=O\left(
|\lambda_{n}|\right)  $ for $n=1,2,\cdots,$ and

(A) $\sum_{n=0}^{\infty}\alpha_{n}\phi_{n}^{k}X_{n}^{k}\left\vert \lambda
_{n}^{k}\right\vert <\infty,$

(B) $\sum_{n=0}^{\infty}\alpha_{n}\phi_{n}^{k-1}{}X_{n}^{k}\left\vert
\Delta\lambda_{n}\right\vert <\infty.$\newline If $A\in\mathcal{S}\left(
\alpha_{n},\phi_{n}\right)  $ satisfies%
\begin{equation}
\sum_{v=0}^{n}\left\vert a_{vv}\widehat{a}_{n,v+1}\right\vert =O\left(
\phi_{n}\right)  , \label{n1}%
\end{equation}%
\begin{equation}
\Delta X_{n}=O\left(  \phi_{n}\right)  ,\ X_{n}=\frac{\phi_{n}}{a_{nn}},
\label{n2}%
\end{equation}
then the summability of $\left\vert A,\alpha_{n}\right\vert _{k}$ for
$k\geq1,$ of the series $\sum C_{n}\left(  t\right)  \lambda_{n}X_{n}$ at any
point is a local property of $f.$
\end{theorem}

\ \ \newline\textbf{Remark 1.} The restrictions of $\left\{  \lambda
_{n}\right\}  $ in Theorem A are relaxed in Theorem 1 to the simple conditions
that $\left\{  \lambda_{n}\right\}  \in BV$ and $\lambda_{n+1}=O\left(
|\lambda_{n}|\right)  ,$ which obviously hold when $\{\lambda_{n}\}$ is a
convex sequence such that $\sum n^{-1}\lambda_{n}$ is convergent.

\begin{theorem}
The result of Theorem 1 also holds when (\ref{n1}) and (\ref{n2}) are replaced
by%
\begin{equation}
\sum_{v=0}^{n}\left\vert \widehat{a}_{n,v+1}\phi_{v}\right\vert =O\left(
\phi_{n}\right)  , \label{n11}%
\end{equation}%
\begin{equation}
\Delta X_{n}=O\left(  n^{-1}\right)  ,\ X_{n}=\frac{1}{n\phi_{n}}%
,n=1,2,\cdots,X_{0}=0, \label{n12}%
\end{equation}
respectively.
\end{theorem}

\ \newline\textbf{Remark 2.} If the matrix $A$ satisfies the condition
$\overline{a}_{n0}=1,\ n=0,1,\cdots,$ then the indexes of the summations in
(A), (B), (\ref{n1}) and (\ref{n11}) only need to run from $1$ instead of $0,$
which can be observed in the proofs of the theorems.

\ \newline\textbf{Remark 3.} Let $\phi_{n}:=a_{nn},\alpha_{n}=\theta_{n}%
^{k-1}.\ $If the matrix $A$ satisfies the conditions in Theorem A, then we can
easily have that $A\in\mathcal{S}\left(  \alpha_{n},\phi_{n}\right)  .$ That
is, Theorem A can be regarded as a corollary of Theorem 2.

\ \

We prove the theorems in Section 2. In Section 3, we show that some well known
matrices such as Cesàro's matrices, Rhaly's generalized Cesàro matrices, the
$p-$Cesàro matrices, and Riesz's matrices are in $\mathcal{S}\left(
\alpha_{n},\phi_{n}\right)  $ for some certain sequences $\left\{  \alpha
_{n}\right\}  $ and $\left\{  \phi_{n}\right\}  ,$ and then derive some new
theorems on the local property of some factored Fourier series, as
applications of the above theorems.

\section{\ \ Proofs of the Main Results}

We prove Theorem 1 in this section. The proof of Theorem 2 is similar.

The behavior of the Fourier series, as far as convergence is concerned, at a
particular value of $x,$ depends on the behavior of the function in the
immediate neighborhood of this point only. Therefore, in order to prove the
theorem, it is sufficient to prove that if $\left\{  s_{n}\right\}  $ is
bounded, then under the conditions of Theorem 1, $\sum a_{n}\lambda_{n}X_{n}$
is summable $\left\vert A,\alpha_{n}\right\vert _{k},\ k\geq1.$ Let $T_{n}$ be
the $n-$th term of the $A-$transform of $\sum_{i=0}^{n}\lambda_{i}a_{i}%
X_{i}.\ $Then
\[
T_{n}=\sum_{v=0}^{n}a_{nv}\sum_{i=0}^{v}a_{i}\lambda_{i}X_{i}=\sum_{i=0}%
^{n}a_{i}\lambda_{i}X_{i}\sum_{v=i}^{n}a_{nv}=\sum_{i=0}^{n}\overline{a}%
_{ni}a_{i}\lambda_{i}X_{i}.
\]
Thus,%
\begin{align*}
T_{n}-T_{n-1}  &  =\sum_{i=0}^{n}\overline{a}_{ni}a_{i}\lambda_{i}X_{i}%
-\sum_{i=0}^{n-1}\overline{a}_{n-1,i}a_{i}\lambda_{i}X_{i}\\
&  =\sum_{i=0}^{n}\widehat{a}_{ni}a_{i}\lambda_{i}X_{i}=\sum_{i=0}%
^{n}\widehat{a}_{ni}\lambda_{i}X_{i}\left(  s_{i}-s_{i-1}\right) \\
&  =\sum_{i=0}^{n-1}\left(  \widehat{a}_{ni}\lambda_{i}X_{i}-\widehat{a}%
_{n,i+1}\lambda_{i+1}X_{i+1}\right)  s_{i}+a_{nn}\lambda_{n}s_{n}X_{n}\\
&  =\sum_{i=0}^{n-1}\widehat{a}_{n,i+1}\Delta\lambda_{i}X_{i}s_{i}+\sum
_{i=0}^{n-1}\widehat{a}_{n,i+1}\lambda_{i+1}\Delta X_{i}s_{i}+\sum_{i=0}%
^{n-1}\left(  \Delta_{i}\widehat{a}_{ni}\right)  \lambda_{i}X_{i}s_{i}\\
&  \ \ \ \ +a_{nn}\lambda_{n}X_{n}s_{n}\\
&  =:T_{n1}+T_{n2}+T_{n3}+T_{n4}.
\end{align*}
Therefore, it is sufficient to prove that%
\begin{equation}
\sum_{n=1}^{\infty}\alpha_{n}\left\vert T_{ni}\right\vert ^{k}<\infty,\text{
for }i=1,2,3,4. \label{s1}%
\end{equation}
Applying Hölder's inequality, we have%

\begin{align*}
\sum_{n=1}^{m+1}\alpha_{n}\left\vert T_{n1}\right\vert ^{k}  &  =O\left(
1\right)  \sum_{n=1}^{m+1}\alpha_{n}\left(  \sum_{i=0}^{n-1}\left\vert
\widehat{a}_{n,i+1}\right\vert \left\vert X_{i}\right\vert \left\vert
\Delta\lambda_{i}\right\vert \right)  ^{k}\\
&  =O\left(  1\right)  \sum_{n=1}^{m+1}\alpha_{n}\left(  \sum_{i=0}%
^{n-1}\left\vert \widehat{a}_{n,i+1}\right\vert \left\vert X_{i}%
^{k}\right\vert \left\vert \Delta\lambda_{i}\right\vert \right)  \left(
\sum_{i=0}^{n-1}\left\vert \widehat{a}_{n,i+1}\right\vert \left\vert
\Delta\lambda_{i}\right\vert \right)  .
\end{align*}
Since $\left\{  \lambda_{n}\right\}  \in BV,$ we have%
\[
\sum_{i=0}^{n-1}\left\vert \widehat{a}_{n,i+1}\right\vert \left\vert
\Delta\lambda_{i}\right\vert =O\left(  \phi_{n}\right)  ,
\]
by (T2). Hence%
\begin{align}
\sum_{n=1}^{m+1}\alpha_{n}\left\vert T_{n1}\right\vert ^{k}  &  =O\left(
1\right)  \sum_{n=1}^{m+1}\alpha_{n}\phi_{n}^{k-1}\sum_{i=0}^{n-1}\left\vert
\widehat{a}_{n,i+1}\right\vert \left\vert X_{i}^{k}\right\vert \left\vert
\Delta\lambda_{i}\right\vert \nonumber\\
&  =O\left(  1\right)  \sum_{i=0}^{m}\left\vert X_{i}^{k}\right\vert
\left\vert \Delta\lambda_{i}\right\vert \sum_{n=i+1}^{m+1}\alpha_{n}\phi
_{n}^{k-1}\left\vert \widehat{a}_{n,i+1}\right\vert \nonumber\\
&  =O\left(  1\right)  \sum_{i=0}^{m}\alpha_{i}\phi_{i}^{k-1}\left\vert
X_{i}^{k}\right\vert \left\vert \Delta\lambda_{i}\right\vert =O\left(
1\right)  , \label{s0}%
\end{align}
by (T3), and (B) of Theorem 1.

It follows from (\ref{n2}) that $\Delta X_{i}=O\left(  a_{ii}X_{i}\right)  .$
Then by Hölder's inequality, (\ref{n1}) and condition (A) of the Theorem 1, we
have%
\begin{align*}
\sum_{n=1}^{m+1}\alpha_{n}\left\vert T_{n2}\right\vert ^{k}  &  =O\left(
1\right)  \sum_{n=1}^{m+1}\alpha_{n}\left(  \sum_{i=0}^{n-1}\left\vert
\widehat{a}_{n,i+1}\lambda_{i+1}\Delta X_{i}\right\vert \right)  ^{k}\\
&  =O\left(  1\right)  \sum_{n=1}^{m+1}\alpha_{n}\left(  \sum_{i=0}%
^{n-1}\left\vert \widehat{a}_{n,i+1}\lambda_{i}a_{ii}X_{i}\right\vert \right)
^{k}\\
&  =O\left(  1\right)  \sum_{n=1}^{m+1}\alpha_{n}\left(  \sum_{i=0}%
^{n-1}\left\vert \widehat{a}_{n,i+1}a_{ii}\right\vert \left\vert \lambda
_{i}^{k}\right\vert \left\vert X_{i}^{k}\right\vert \right)  \left(
\sum_{i=0}^{n-1}\left\vert \widehat{a}_{n,i+1}a_{ii}\right\vert \right)
^{k-1}\\
&  =O\left(  1\right)  \sum_{n=1}^{m+1}\alpha_{n}\phi_{n}^{k-1}\left(
\sum_{i=0}^{n-1}\left\vert \widehat{a}_{n,i+1}a_{ii}\right\vert \left\vert
\lambda_{i}^{k}\right\vert \left\vert X_{i}^{k}\right\vert \right) \\
&  =O\left(  1\right)  \sum_{i=0}^{m}\left\vert \lambda_{i}^{k}\right\vert
\left\vert X_{i}^{k}\right\vert \left\vert a_{ii}\right\vert \sum
_{n=i+1}^{m+1}\alpha_{n}\phi_{n}^{k-1}\left\vert \widehat{a}_{n,i+1}%
\right\vert \\
&  =O\left(  1\right)  \sum_{i=0}^{m}\alpha_{n}\phi_{n}^{k-1}\left\vert
\lambda_{i}^{k}\right\vert \left\vert X_{i}^{k}\right\vert \left\vert
a_{ii}\right\vert \\
&  =O\left(  1\right)  \sum_{i=0}^{m}\alpha_{n}\phi_{n}^{k}\left\vert
\lambda_{i}^{k}\right\vert \left\vert X_{i}^{k}\right\vert =O\left(  1\right)
,
\end{align*}
where we also used the fact that $\widehat{a}_{nn}=a_{nn}=O\left(  \phi
_{n}\right)  ,$ which follows from (T2).

By (T1), (T3) and condition (A), we have%
\begin{align}
\sum_{n=1}^{m+1}\alpha_{n}\left\vert T_{n3}\right\vert ^{k}  &  =O\left(
1\right)  \sum_{n=1}^{m+1}\alpha_{n}\left(  \sum_{i=0}^{n-1}\left\vert
\Delta\widehat{a}_{n,i+1}\lambda_{i}X_{i}\right\vert \right)  ^{k}\nonumber\\
&  =O(1)\sum_{n=1}^{m+1}\alpha_{n}\left(  \sum_{i=0}^{n-1}\left\vert
\Delta\widehat{a}_{n,i+1}\right\vert \left\vert \lambda_{i}^{k}\right\vert
\left\vert X_{i}^{k}\right\vert \right)  \left(  \sum_{i=0}^{n-1}\left\vert
\Delta\widehat{a}_{n,i+1}\right\vert \right)  ^{k-1}\nonumber\\
&  =O\left(  1\right)  \sum_{n=1}^{m+1}\alpha_{n}\phi_{n}^{k-1}\sum
_{i=0}^{n-1}\left\vert \Delta\widehat{a}_{n,i+1}\right\vert \left\vert
\lambda_{i}^{k}\right\vert \left\vert X_{i}^{k}\right\vert \nonumber\\
&  =O\left(  1\right)  \sum_{i=0}^{m}\left\vert \lambda_{i}^{k}\right\vert
\left\vert X_{i}^{k}\right\vert \sum_{n=i+1}^{m+1}\alpha_{n}\phi_{n}%
^{k-1}\left\vert \Delta\widehat{a}_{n,i+1}\right\vert \nonumber\\
&  =O\left(  1\right)  \sum_{i=0}^{m}\alpha_{i}\phi_{i}^{k}\left\vert
\lambda_{i}^{k}\right\vert \left\vert X_{i}^{k}\right\vert =O\left(  1\right)
. \label{s3}%
\end{align}

By using $a_{nn}=O\left(  \phi_{n}\right)  $ again, we have%
\begin{align}
\sum_{n=1}^{m+1}\alpha_{n}\left\vert T_{n4}\right\vert ^{k}  &  =O\left(
1\right)  \sum_{n=1}^{m+1}\alpha_{n}\left\vert a_{nn}\lambda_{n}%
X_{n}\right\vert ^{k}\nonumber\\
&  =O(1)\sum_{n=1}^{m+1}\alpha_{n}\phi_{n}^{k}\left\vert \lambda_{n}%
^{k}\right\vert \left\vert X_{n}^{k}\right\vert \nonumber\\
&  =O\left(  1\right)  . \label{s4}%
\end{align}
Combining (\ref{s0})-(\ref{s4}), we have (\ref{s1}). This proves Theorem 1.

\section{ Applications of The Theorems}

\subsection{Cesàro's Matrices}

We will use the following formula often in the proofs (see \cite{Z} for proof,
for example):%

\begin{equation}
A_{n}^{\alpha}=\frac{n^{\alpha}}{\Gamma\left(  \alpha+1\right)  }\left(
1+O\left(  \frac{1}{n}\right)  \right)  . \label{n}%
\end{equation}
In this subsection, we set%

\[
\phi_{0}:=1,\;\phi_{n}:=\left\{
\begin{array}
[c]{cc}%
n^{-1}, & \alpha>1\\
\frac{1}{A_{n}^{\alpha}}=c_{nn}, & 0<\alpha\leq1,
\end{array}
\right.  n=1,2,\cdots.
\]
By (\ref{n}), we see that%
\[
\phi_{n}\sim\left\{
\begin{array}
[c]{cc}%
n^{-1}, & \alpha>1\\
n^{-\alpha}, & 0<\alpha\leq1,
\end{array}
\right.  n=1,2,\cdots.
\]

Recall that a nonnegative sequence $\left\{  a_{n}\right\}  $ is said to be
almost decreasing, if there is a positive constant $K$ such that%
\[
a_{n}\geq Ka_{m}%
\]
holds for all $n\leq m,\ $and it is said to be quasi-$\beta$-power decreasing,
if $\left\{  n^{\beta}a_{n}\right\}  $ is almost decreasing.

It should be noted that every decreasing sequence is an almost decreasing
sequence, and every almost decreasing sequence is a quasi-$\beta$-power
decreasing sequence for any non-positive index $\beta,$ but the converse is
not true.

\begin{lemma}
(\cite{YZ})Let $\alpha>0,\ $and let $\left\{  \alpha_{n}\right\}  $ be a
sequence of positive numbers. If $\left\{  \alpha_{n}\phi_{n}^{k-1}%
n^{-1}\right\}  $ is quasi-$\varepsilon$-power decreasing for some
$\varepsilon>0,$ then $C\in\mathcal{S}\left(  \alpha_{n},\phi_{n}\right)  .$
\end{lemma}

A direct calculation leads to
\begin{align*}
\widehat{c}_{ni}  &  =\frac{1}{A_{n}^{\alpha}}\sum_{j=i}^{n}A_{n-j}^{\alpha
-1}-\frac{1}{A_{n-1}^{\alpha}}\sum_{j=i}^{n-1}A_{n-1-j}^{\alpha-1}\\
&  =\frac{A_{n-i}^{\alpha}}{A_{n}^{\alpha}}-\frac{A_{n-1-i}^{\alpha}}%
{A_{n-1}^{\alpha}}=\frac{iA_{n-i}^{\alpha-1}}{nA_{n}^{\alpha}}.
\end{align*}
Thus, for $0<\alpha\leq1,$
\begin{align}
\sum_{v=0}^{n}\left\vert c_{vv}\widehat{c}_{n,v+1}\right\vert  &  =O\left(
\frac{1}{n^{1+\alpha}}\right)  \sum_{v=1}^{n}\frac{\left(  v+1\right)
A_{n-v-1}^{\alpha-1}}{A_{v}^{\alpha}}\nonumber\\
&  =O\left(  \frac{1}{n^{2}}\right)  \sum_{v=1}^{n/2}v^{1-\alpha}+O\left(
\frac{1}{n^{2\alpha}}\right)  \sum_{v=1}^{n/2}v^{\alpha-1}\nonumber\\
&  =O\left(  \phi_{n}\right)  . \label{l23}%
\end{align}
Similarly, for $\alpha>1,$
\begin{equation}
\sum_{v=1}^{n}\frac{1}{v}\left\vert \widehat{c}_{n,v+1}\right\vert =O\left(
\phi_{n}\right)  . \label{l24}%
\end{equation}
Now set%
\[
X_{n}\equiv1=\left\{
\begin{array}
[c]{cc}%
\frac{\phi_{n}}{c_{nn}}, & 0<\alpha\leq1,\\
(n\phi_{n})^{-1}, & \alpha>1.
\end{array}
\right.
\]
Then $X_{n}$ satisfies (\ref{n2}) and (\ref{n12}) for $0<\alpha\leq1$ and
$\alpha>1$ respectively. Now, applying Lemma 1, (\ref{l23}), (\ref{l24}),
Theorem 1 and Theorem 2, we have the following

\begin{theorem}
Let $\alpha>0,\ \left\{  \alpha_{n}\right\}  \ $be sequences of positive
numbers$.\ $Let$\ \left\{  \lambda_{n}\right\}  \in BV$ be a sequence of
complex numbers such that \ $\lambda_{n+1}=O\left(  |\lambda_{n}|\right)  $
for $n=1,2,\cdots,$ and

(a) $\sum_{n=1}^{\infty}\alpha_{n}\phi_{n}^{k}\left\vert \lambda_{n}%
^{k}\right\vert <\infty,$

(b) $\sum_{n=1}^{\infty}\alpha_{n}\phi_{n}^{k-1}{}\left\vert \Delta\lambda
_{n}\right\vert <\infty.$ \newline If $\left\{  \alpha_{n}\phi_{n}^{k-1}%
n^{-1}\right\}  $ is quasi-$\varepsilon$-power decreasing, then the
summability of $\left\vert C,\alpha_{n}\right\vert _{k}$ for $k\geq1,$ of the
series $\sum C_{n}\left(  t\right)  \lambda_{n}$ at any point is a local
property of $f.$
\end{theorem}

As examples, we give two corollaries of Theorem 3.

\begin{corollary}
Let$\ \left\{  \lambda_{n}\right\}  \in BV$ be a sequence of complex numbers
such that \ $\lambda_{n+1}=O\left(  |\lambda_{n}|\right)  $ for $n=1,2,\cdots
,$ and

(c) $\sum_{n=1}^{\infty}n^{\delta k-1}\log^{\gamma}n\left\vert \lambda_{n}%
^{k}\right\vert <\infty,$

(d) $\sum_{n=1}^{\infty}n^{\delta k}\log^{\gamma}n\left\vert \Delta\lambda
_{n}\right\vert <\infty,$ \newline then the summability of $\left\vert
C,n^{\delta k+k-1}\log^{\gamma}n\right\vert _{k},$ for$\ \alpha\geq
1,\ \gamma\in R,\ k\geq1\ $and $0\leq\delta<\frac{1}{k},$ of the series $\sum
C_{n}\left(  t\right)  \lambda_{n}$ at any point is a local property of $f.$
\end{corollary}

\begin{proof}
Let $\alpha_{n}=n^{\delta k+k-1}log^{\gamma}n,\ n=1,2,\cdots,$ $\alpha_{0}=1.$
Since $\alpha\geq1,$ then $\phi_{n}=n^{-1}.$\ It is then obvious that (c)
implies (a), and (d) implies (b).\ From the condition that $0\leq\delta
<\frac{1}{k},$ we see that there exists an $\varepsilon>0$ such that $\delta
k-1+\varepsilon<0,$ and thus $\left\{  n^{\delta k-1+\varepsilon}\log^{\gamma
}n\right\}  \ $is quasi decreasing for $\gamma\in\mathbb{R}.$ In other words,
$\left\{  \alpha_{n}\phi_{n}^{k-1}n^{-1}\right\}  $ is quasi-$\varepsilon
$-power decreasing. Therefore, by Theorem 3, we have Corollary 1.
\end{proof}

\begin{corollary}
Let$\ \left\{  \lambda_{n}\right\}  \in BV$ be a sequence of complex numbers
such that \ $\lambda_{n+1}=O\left(  |\lambda_{n}|\right)  $ for $n=1,2,\cdots
,$ and

(c$^{\prime}$) $\sum_{n=1}^{\infty}n^{\delta k+\left(  1-\alpha\right)
k-1}\log^{\gamma}n\left\vert \lambda_{n}^{k}\right\vert <\infty,$

(d$^{\prime}$) $\sum_{n=1}^{\infty}n^{\delta k+\left(  1-\alpha\right)
\left(  k-1\right)  }\log^{\gamma}nX_{n}\left\vert \Delta\lambda
_{n}\right\vert <\infty,$ \newline then the summability of $\left\vert
C,n^{\delta k+k-1}\log^{\gamma}n\right\vert _{k},$ for$\ 0<\alpha
<1,\ \gamma\in R,\ k\geq1\ $and $0\leq\delta<\frac{2-\alpha+\left(
1-\alpha\right)  k}{k},$ of the series $\sum C_{n}\left(  t\right)
\lambda_{n}$ at any point is a local property of $f.$
\end{corollary}

\begin{proof}
Note that $\phi_{n}=n^{-\alpha}$ for $0<\alpha<1.$ Then the proof of Corollary
2 is similar to that of Corollary 1.
\end{proof}

\subsection{ Rhaly's Generalized Cesàro Matrices}

Let $D$ be the Rhaly generalized Cesàro matrix (see \cite{R1}), that is, $D$
has entries of the form $d_{nk}=t^{n-k}/\left(  n+1\right)  ,\ k=0,1,\cdots
,n,\ n=1,2,\cdots.$ When $t=1,\ $the Rhaly generalized Cesàro matrix reduces
to the Cesàro matrix of order $1.\ $We shall restrict our attention to
$0<t<1.$ In this case, $D$ does not satisfy condition (ii) of Theorem A. It is
routine to deduce that%
\begin{equation}
\widehat{d}_{nv}=\sum_{r=v}^{n}\frac{t^{n-r}}{n+1}-\sum_{r=v}^{n-1}%
\frac{t^{n-1-r}}{n}=\frac{1}{1-t}\left(  \frac{1-t^{n-v+1}}{n+1}%
-\frac{1-t^{n-v}}{n}\right)  . \label{r1}%
\end{equation}
Set $\phi_{0}=1,\;\phi_{n}=n^{-1},n=1,2,\cdots.\ $By (\ref{r1}), we have
\begin{align*}
\widehat{d}_{nv}  &  =\frac{1}{1-t}\left(  \frac{1-t^{n-v+1}}{n+1}%
-\frac{1-t^{n-v}}{n}\right) \\
&  =-\frac{1-t^{n-v}-nt^{n-v}\left(  1-t\right)  }{\left(  1-t\right)
n\left(  n+1\right)  }\\
&  =O\left(  \frac{1}{n\left(  n+1\right)  }+\frac{nt^{n-v}}{n\left(
n+1\right)  }\right)  .
\end{align*}
Thus%
\begin{equation}
\sum_{v=0}^{n}\left\vert d_{vv}\widehat{d}_{n,v+1}\right\vert =O\left(
\frac{1}{n^{2}}\right)  \sum_{v=1}^{n}\frac{1}{v}+O\left(  \frac{1}{n}\right)
\sum_{v=1}^{n}t^{n-v}=O\left(  \phi_{n}\right)  . \label{l25}%
\end{equation}

\begin{lemma}
(\cite{YZ})Let $0<t<1,\ $and $\left\{  \alpha_{n}\right\}  $ be a sequences of
positive numbers. If $\left\{  \alpha_{n}\phi_{n}^{k-1}n^{-1}\right\}  $ is
quasi-$\varepsilon-$power decreasing for some $\varepsilon>0,$ then
$D\in\mathcal{S}\left(  \alpha_{n},\phi_{n}\right)  .$
\end{lemma}

Now set%

\[
X_{n}=\frac{\phi_{n}}{d_{nn}}=\frac{n+1}{n}.
\]
Then $X_{n}=O\left(  1\right)  $ and $\Delta X_{n}=O\left(  \phi_{n}\right)
.$ Therefore, by Lemma 2, (\ref{l25}) and Theorem 1, we have

\begin{theorem}
Let $0<t<1$\ and let $\left\{  \alpha_{n}\right\}  \ $be a sequence of
positive numbers$.\ $Assume that$\ \left\{  \lambda_{n}\right\}  \in BV$ is a
sequence of complex numbers such that \ $\lambda_{n+1}=O\left(  |\lambda
_{n}|\right)  $ for $n=1,2,\cdots,$ and (A), (B) in Theorem 1 hold. If
$\left\{  \alpha_{n}\phi_{n}^{k-1}n^{-1}\right\}  $ is quasi-$\varepsilon
$-power decreasing for some $\varepsilon>0$, then the summability of
$\left\vert D,\alpha_{n}\right\vert _{k}$ for $k\geq1,$ of the series $\sum
C_{n}\left(  t\right)  \frac{n+1}{n}\lambda_{n}$ at any point is a local
property of $f.$
\end{theorem}

Obviously, we can also have a corollary of Theorem 3 that is similar to
Corollary 1. We omit the details here.

\subsection{$p-$Cesáro Matrices}

Let $E$ be the $p-$Cesàro matrix (see \cite{R2}), that is, the entries of $E$
has the form $e_{ni}=1/\left(  n+1\right)  ^{p},\ i=0,1,\cdots
,n,\ n=1,2,\cdots.$ When $p=1,$ the $p-$Cesàro matrix reduces to the Cesàro
matrix of order $1$ again. Also, $E$ does not satisfy condition (ii) of
Theorem A. We restrict our attention to the case when $1<p\leq2.$

Set $\phi_{0}=1,\;\phi_{n}=n^{-p},n=1,2,\cdots.$ Then%
\begin{equation}
\widehat{e}_{ni}=\overline{e}_{ni}-\overline{e}_{n-1,i}=\frac{n-i+1}{\left(
n+1\right)  ^{p}}-\frac{\left(  n-i\right)  }{n^{p}}, \label{p1}%
\end{equation}
and%
\begin{equation}
\Delta_{i}\widehat{e}_{ni}=\widehat{e}_{n,i+1}-\widehat{e}_{ni}=e_{ni}%
-e_{n-1,i}=\frac{1}{\left(  n+1\right)  ^{p}}-\frac{1}{n^{p}}. \label{p2}%
\end{equation}
By (\ref{p1}), we have
\begin{equation}
\widehat{e}_{ni}=\left(  n-i\right)  \left(  \frac{1}{\left(  n+1\right)
^{p}}-\frac{1}{n^{p}}\right)  +\frac{1}{\left(  n+1\right)  ^{p}}=O\left(
\phi_{n}\right)  . \label{p3}%
\end{equation}
Now set $X_{n}=\frac{\phi_{n}}{e_{n}n}$. Then direct calculations yield that
\[
\Delta X_{n}=O\left(  n^{-2}\right)  =O(\phi_{n}),\ 1<p\leq2,
\]
and
\[
\sum_{v=1}^{n}\left\vert e_{vv}\widehat{e}_{n,v+1}\right\vert =O\left(
\phi_{n}\right)  .
\]

\begin{lemma}
(\cite{YZ})Let $p>1\ $and $\left\{  \alpha_{n}\right\}  $ be a sequences of
positive numbers. If $\left\{  \alpha_{n}\phi_{n}^{k-1}n^{-1}\right\}  $ is
quasi-$\varepsilon$-power decreasing for some $\varepsilon>0$ such that
$p-2+\varepsilon>0,$ then $D\in\mathcal{S}\left(  \alpha_{n},\phi_{n}\right)
.$
\end{lemma}

Therefore, we have

\begin{theorem}
Let $1<p\leq2$ and let$\ \left\{  \alpha_{n}\right\}  \ $be a sequence of
positive numbers$.\ $Assume that$\ \left\{  \lambda_{n}\right\}  \in BV$ is a
sequence of complex numbers such that \ $\lambda_{n+1}=O\left(  |\lambda
_{n}|\right)  $ for $n=1,2,\cdots,$ and (A), (B) in Theorem 1 hold. If
$\left\{  \alpha_{n}\phi_{n}^{k-1}n^{-1}\right\}  $ is quasi-$\varepsilon
$-power decreasing for some $\varepsilon>0\ $such that $p-2+\varepsilon>0,$
then the summability of $\left\vert E,\alpha_{n}\right\vert _{k}$ for
$k\geq1,$ of the series $\sum C_{n}\left(  t\right)  X_{n}\lambda_{n}$ at any
point is a local property of $f.$
\end{theorem}

\subsection{Riesz's matrices}

We firstly establish a general result, then apply it to the Riesz's matrices.

\begin{lemma}
(\cite{YZ}) \textit{Let }$A$\textit{ be a lower triangular matrix with
nonnegative entries, and }$\left\{  \alpha_{n}\right\}  $\textit{ be a
sequence of positive numbers. If}
\end{lemma}

(I) $\overline{a}_{n0}=1,\ n=0,1,\cdots,$

(II) $a_{n-1,v}\geq a_{nv}$ for $n\geq v+1,$

(III) $na_{nn}=O\left(  1\right)  ,$

(IV) $\sum_{n=v+1}^{\infty}\alpha_{n}n^{-k+1}\left\vert \Delta_{v}%
\widehat{a}_{nv}\right\vert =O\left(  \alpha_{v}a_{vv}v^{-k+1}\right)  ,$

(V) $\sum_{n=v+1}^{\infty}\alpha_{n}n^{-k+1}\widehat{a}_{n,v+1}=O\left(
\alpha_{v}v^{-k+1}\right)  ,$ \newline then\textit{ }$A\in\mathcal{S}\left(
\alpha_{n},n^{-1}\right)  .$

Now, by setting $\phi_{0}=1,X_{0}=0,\phi_{n}:=n^{-1},X_{n}=\left(  n\phi
_{n}\right)  ^{-1},n=1,2,\cdots,$ and applying Theorem 2, we have

\begin{theorem}
\textit{Let }$\left\{  \alpha_{n}\right\}  $\textit{ be a sequence of positive
numbers, and let }$A$\textit{ be a lower triangular matrix with nonnegative
entries satisfying conditions (I)-(V) of Lemma 4. }Assume that$\ \left\{
\lambda_{n}\right\}  \in BV$ is a sequence of complex numbers such that
\ $\lambda_{n+1}=O\left(  |\lambda_{n}|\right)  $ for $n=1,2,\cdots,$ and (A),
(B) in Theorem 1 hold. If (\ref{n11}) and (\ref{n12}) hold, then the
summability of $\left\vert A,\alpha_{n}\right\vert _{k}$ for $k\geq1,$ of the
series $\sum C_{n}\left(  t\right)  X_{n}\lambda_{n}$ at any point is a local
property of $f.$
\end{theorem}

We now show that under some necessary conditions, Riesz matrix $R$ satisfies
all the conditions in Lemma 4. For any positive sequence $\left\{
p_{n}\right\}  $ such that $P_{n}=p_{0}+p_{1}+\cdots+p_{n}\rightarrow\infty$,
the corresponding Riesz matrix $R$ has the entries $r_{nv}:=\frac{p_{v}}%
{P_{n}},\,v=0,1,\cdots,n;\,n=0,1,2,\cdots$ . Now obviously, we have
$\overline{r}_{n0}=1$ and $r_{n-1,v}\geq r_{nv}$ for $n\geq v+1$. Direct
calculations yield that (set $P_{-1}=0$)%
\begin{equation}
\widehat{r}_{nv}=\frac{P_{v-1}p_{n}}{P_{n}P_{n-1}}, \label{eq:a1}%
\end{equation}
and%
\begin{equation}
\left\vert \Delta_{v}\widehat{r}_{nv}\right\vert =\frac{p_{n}p_{v}}%
{P_{n}P_{n-1}}. \label{eq:a2}%
\end{equation}
So if $np_{n}=O\left(  P_{n}\right)  $ and
\begin{equation}
\sum_{n=v+1}^{m+1}\alpha_{n}n^{-k+1}\frac{p_{n}}{P_{n}P_{n-1}}=O\left(
\alpha_{v}v^{-k+1}P_{v}^{-1}\right)  , \label{eq:m1}%
\end{equation}
then $R$ satisfies all conditions in Lemma 4.

Thus, we have (note that $X_{n}:=\left(  n\phi_{n}\right)  ^{-1}$)

\begin{theorem}
Let $\left\{  p_{n}\right\}  $ be a positive sequence satisfying
$P_{n}\rightarrow\infty$, $np_{n}=O(P_{n})$ and (\ref{eq:m1}). Assume
that$\ \left\{  \lambda_{n}\right\}  \in BV$ is a sequence of complex numbers
such that \ $\lambda_{n+1}=O\left(  |\lambda_{n}|\right)  $ for $n=1,2,\cdots
,$ and (A), (B) in Theorem 1 hold. \textit{If (\ref{n11}) and (\ref{n12})
hold, then }the summability of $\left\vert R,\alpha_{n}\right\vert _{k}$ for
$k\geq1,$ of the series $\sum C_{n}\left(  t\right)  X_{n}\lambda_{n}$ at any
point is a local property of $f.$
\end{theorem}

\end{document}